\documentclass[11pt]{article}
\usepackage{amsthm,amsmath,amssymb}

\newcommand{\subject}[1]{\begin{flushleft}\textbf{2000 MR  Subject Classification}: #1\end{flushleft}}
\newcommand{\university}[1]{\\[3mm]{\small #1}}

\newtheorem{defi}{Definition}[section]

\newtheorem{theorem}[defi]{Theorem}
\newtheorem{cor}[defi]{Corollary}

\newtheorem{exam}[defi]{Example}

\title{Some General New Einstein Walker Manifolds }
\author{Mehdi Nadjafikhah
\thanks{Corresponding author. School of Mathematics, Iran University of Science and Technology, Narmak, Tehran 1684613114, Iran. E-mail: m\underline{ }nadjafikhah@iust.ac.ir}
\university{}
 \and
 Mehdi Jafari
 \thanks{Department of Complementary Education, Payame Noor University, PO BOX 19395-3697, Tehran, Iran. E-mail: m.jafari@phd.pnu.ac.ir}
 \university{} }
 \date{}
\begin{document}
\maketitle
\hspace*{-6mm}\textbf{Abstract}  In this paper, Lie symmetry group method is applied to find the lie point symmetries group of a PDE system that is determined general form of four-dimensional Einstein Walker manifold. Also we will construct the optimal system of one-dimensional Lie subalgebras and investigate some of its group invariant solutions.
\\

\hspace*{-6mm}\textbf{Keywords}  {Lie symmetry group, invariant solution, optimal system of Lie sub-algebras, Walker manifold}
 \subject{70G65, 34C14, 53C50}
\section{Introduction}
As it is well known, the symmetry group method has an irrefragable role in the analysis of differential equations. The theory of Lie symmetry groups of differential equations was developed by Sophus Lie \cite{Lie}. By this method we can construct  new solutions from known ones, reduce the order of ODEs and investigate the invariant solutions (for more information about the other applications of Lie symmetries, see\cite{Olv1},\cite{Blucol} and \cite{BluKum}).
Lei's method led to an algorithmic approach to find special solution of differential equation by it's symmetry group. These solutions are obtained by solving the reduced system of differential equation having less independent variables than the original system. These solutions are called group invariant solutions. Bluman and Cole generalized the Lie's method for finding the group-invariant solutions\cite{BluCol}.
In this paper we apply this method to find the invariant solutions of a system of PDEs that determines spacial kind of Walker manifolds.
\subsubsection*{Pseudo-Riemannian manifolds}
Let $\langle\cdot,\cdot\rangle$ be a non-degenerate inner product on a vector space $V$. We can choose  a basis $\{e_i\}$ for $V$ so that
\begin{equation*}
\langle e_i ,e_j\rangle = \left\{
  \begin{array}{lcr} 0 &&  {i \ne j}  \\  \pm 1 && i = j  \\ \end{array} \right.
\end{equation*}
such a basis is called an orthonormal basis. We set $\varepsilon_i:=\langle e_i ,e_i\rangle$. Let $p$ be the number of indices $i$ with $\varepsilon_i=-1$ and  $q=dimV-p$ be the complementary index. The inner product is then said to have signature $(p,q)$; the integers $p$ and $q$ are independent of the particular orthonormal basis chosen. We can extend this argument  on manifolds. Let $\mathcal{M}=(M,g)$ where $M$ is a $m$ dimensional  manifold and $g$ is a symmetric non-degenerate smooth bilinear form on $TM$ of signature $(p,q)$, such a manifold called a pseudo-Riemannian manifold of signature $(p,q)$. So we have $p+q=m$. If $p=0$ i.e. , if $g$ is positive definite, we say $\mathcal{M}$ is a Riemannian manifold. If $(x_1,...,x_m)$ is a system of local coordinates on $M$, we may express $g=\sum_{i,j}g_{ij}dx_i\circ dx_j$ where $g_{ij}:=g(\partial_{x_i},\partial_{x_j})$ and $ {``}\circ ^{^{\textbf{,,}}}$ is symmetric product.
\subsubsection*{Walker manifolds}
Let $M$ be a pseudo-Riemannian manifold of signature $(p,q)$. We suppose given  a splitting of the tangent bundle in the form $TM=V_1\oplus V_2$ where $V_1$ and $V_2$  are smooth subbundles which are called distribution. This defines two complementary projection $\pi_1$ and $\pi_2$  of $TM$ onto  $V_1$ and $V_2$. We say that $V_1$ is a parallel distribution if $\nabla\pi_1=0$. Equivalently this means that if $X_1$ is any smooth vector field taking values in $V_1$, then $\nabla X_1$ again takes values in $V_1$. If $M$ is Riemannian, we can take $V_2=V^\bot_1$ to be the orthogonal complement of $V_1$ and in that case $V_2$ is again parallel. In the pseudo-Riemannian setting, $V_1\cap V_2$ need not be trivial. We say that $V_1$ is a null parallel distribution if $V_1$ is parallel and if the metric restricted to $V_1$ vanishes identically. Manifolds which admit null parallel distributions are called Walker manifolds.
 Walker showed that a canonical form of a $2n$-dimensional pseudo-Riemannian Walker manifold which admit a $n$-dimensional distribution is given by the metric tensor :
\[(g_{ij} ) = \left( {\begin{array}{*{20}c} 0 & {Id_n }  \\ {Id_n } & B  \\
\end{array} } \right)\]
where $Id_n$ is the $n \times n$ identity matrix and $B$ is a symmetric $n \times n$ matrix whose entries are functions of the $(x_1 ,...,x_{2n} )$. (for more details see \cite{Gar} and \cite{Walk}).

 If we assume that $n = 2$ we adopt the following result for 4-dimensional Walker manifolds;
 Let $M_{a,b,c}:=(\mathcal{O},g_{a,b,c})$, where $\mathcal{O}$ be an open subset of $\mathbb{R}^4$ and ${a,b,c}\in C^\infty(\mathcal{O})$ be smooth functions on $\mathcal{O}$, then
 \begin{equation*}
(g_{a,b,c} )_{ij}  = \left( \begin{array}{cccc}
   0 & 0 & 1 & 0  \\
   0 & 0 & 0 & 1  \\
   1 & 0 & a & c  \\
   0 & 1 & c & b
 \end{array}  \right)
 \end{equation*}
 So we can express the general form of 4-dimensional Walker manifolds as follow,
 \begin{eqnarray}
 \begin{array}{lclcl}
 g_{a,b,c}:=2(dx_1\circ dx_3+dx_2\circ  dx_4)+a(x_1,x_2,x_3,x_4)dx_3\circ dx_3\vspace*{2mm}\\
 \hspace{13mm}+b(x_1,x_2,x_3,x_4)dx_4\circ dx_4+2c(x_1,x_2,x_3,x_4)dx_3\circ dx_4.
 \end{array}
 \end{eqnarray}
 \subsubsection*{Einstein Walker manifolds}
 Let $D$ be a connection on manifold $M$. The curvature operator $\mathcal{R}$ is defined by the formula
\begin{equation*}
\mathcal{R}(X,Y)Z:=(D_XD_Y-D_YD_X-D_{[X,Y]})Z
\end{equation*}
The associated Ricci tensor $\rho$ is defined by
\begin{equation*}
\rho(X,Y):=\mathsf{Tr}\{Z\longrightarrow\mathcal{R}(Z,X)Y\}.
\end{equation*}
A Walker manifold is said to be Einstein Walker manifold if  its Ricci tensor is a scaler multiple of the metric at each point i.e. , there is a constant $c$ so that $\rho=cg$. We can see that $M_{a,b,c}$ is Einstein if and only if the functions $a$, $b$ and $c$ verify the  following system of PDEs (\cite{Gar}, page 81).
$$a_{11}-b_{22}=0,\hspace{2cm} b_{12}+c_{11}=0,\hspace{2cm}a_{12}+c_{22}=0,$$\vspace{-6mm}
$$\hspace{1mm}a_{1}c_{2}+a_{2}b_{2}-a_{2}c_{1}-{c_{2}}^{2}+2\,c_{}a_
{12}+b_{}a_{22}-2\,a_{24}-a_{}c_{12}+2c_{23}=0,$$
$$\hspace{2mm}a_{2}b_{1}-c_{1}c_{2}+c_{}a_{11}-a_{14}-b_{23}-a_{}c_{11}-c_{}c_{12}+c_{13}-b_{}c_{22}+c_{24}=0,$$ $$\hspace{-1mm}a_{1}b_{1}-b_{1}c_{2}+b_{2}c_{1}-{c_{1}}^{2}+a_{}b_{11}+2\,c_{}b_{12}-2b_{13}-b_{}c_{12}+2\,c_{14}=0.$$
This system is hard to handle, for this reason we consider the spacial case in this paper. Assume that the functions $a$, $b$ and $c$ only depend on $x_1$ and $x_2$, therefore we must solve this system,
$$\hspace{-3mm}a_{11}-b_{22}=0,\hspace{8mm} b_{12}+c_{11}=0,\hspace{8mm}a_{12}+c_{22}=0,$$\vspace{-6mm}
$$\hspace{-3mm}a_{1}c_{2}+a_{2}b_{2}-a_{2}c_{1}-{c_{2}}^{2}+2\,c_{}a_
{12}+b_{}a_{22}-a_{}c_{12}=0,$$\vspace{-7.25mm}
\begin{eqnarray}
\hspace{-9mm}a_{2}b_{1}-c_{1}c_{2}+c_{}a_{11}-a_{}c_{11}-c_{}c_{12}-b_{}c_{22}=0,
\end{eqnarray}\vspace{-6mm}
 $$\hspace{-4mm}a_{1}b_{1}-b_{1}c_{2}+b_{2}c_{1}-{c_{1}}^{2}+a_{}b_{11}+2\,c_{}b_{12}-b_{}c_{12}=0.$$
This work is organized as follows. In section 2, some preliminary results about Lie symmetry method is presented. In section 3, the infinitesimal generators of symmetry algebra of system (2) are determined and some results obtained. In section 4, the optimal system of sub-algebras is constructed. In section 5, we obtain the invariant solutions corresponding to the infinitesimal  symmetries of system (2).
\section{Method of Lie symmetries}
In this section, we will perform the procedure of determining symmetries of a system of PDEs (see \cite{Olv1} and \cite{Nad1}).
To begin, we consider a general case for a system of PDE of order $n$th with $p$ independent and $q$ dependent variables such as:
\begin{eqnarray}
\Delta_{\mu}(x,u^{(n)})=0,    \ \ \ \ \ \ \ \ \          \mu=1,...,r,
\end{eqnarray}
involving $x=(x^{1},...,x^{p})$ and $u=(u^{1},...,u^{q})$ as independent and dependent variables respectively and all the derivatives of $u$ with respect to $x$ from $0$ to $n$.
We consider a one parameter Lie group of infinitesimal transformations who act on both dependent and independent variables of (3):
\begin{eqnarray}
\tilde {x}^{i}=x^i+\varepsilon\xi^i(x,u)+o(\varepsilon^2),\ \ \ \ \ \ i=1,...,p,\\
\tilde {u}^{j}=u^j+\varepsilon\phi_j(x,u)+o(\varepsilon^2),\ \ \ \ \ j=1,...,q,\nonumber
\end{eqnarray}
where $\xi^i$ and $\phi^j$ are the infinitesimals of the transformations for the independent and dependent variables, respectively.
The general form of infinitesimal generator associated with the above group of transformations is
\begin{eqnarray}
X = \sum\limits_{i = 1}^p {\xi ^i (x ,u )
{{\partial_{ x^i} }}}  + \sum\limits_{j = 1}^q {\phi_j  (x ,u )
{{\partial_{u^j}}}}
\end{eqnarray}
Transformations which maps solutions of a differential equation to other solutions are called symmetries of this equation. We define the $n$th order prolongation of the infinitesimal generator $X$ by
\begin{eqnarray}
Pr^{(n)}X=X+\sum\limits_{\alpha= 1}^q\sum\limits_{J}{\phi}^{J}_{\alpha}(x,u^{(n)})\partial_{u^{\alpha}_{J}}
\end{eqnarray}
where $J=(j_{1},...,j_{k})$, $1\leq j_k\leq p$, $1\leq k\leq n$ and the sum is over all $J$'s of order $0<\#J\leq n$. If $\#J=k$, the coefficient ${\phi}^{J}_{\alpha}$ of $\partial_{u^{\alpha}_J}$ will only depend on $k$-th and lower order derivatives of $u$, and
\begin{eqnarray}
{\phi}^{J}_{\alpha}(x,u^{(n)})=D_J(\phi_{\alpha}-\sum\limits_{i= 1}^p\xi^{i}u^{\alpha}_{i})+\sum\limits_{i=1}^p\xi^{i}u^{\alpha}_{J,i}
\end{eqnarray}
where $u^{\alpha}_{i}:=\frac{\partial u^{\alpha}}{\partial x^{i}}$ and $u^{\alpha}_{J,i}:=\frac{\partial u^{\alpha}_{J}}{\partial x^{i}}$.
Now, according to theorem 2.36 of \cite{Olv1}, the invariance of the system (3) under the infinitesimal transformations leads to the invariance conditions:
 \begin{eqnarray}
Pr^{(n)}X[\Delta_{\mu}(x,u^{(n)})]=0,    \ \ \   \mu=1,...,r, \  \ \mathrm{whenever} \  \  \Delta_{\mu}(x,u^{(n)})=0.
\end{eqnarray}
It is important that the infinitesimal symmetries form a Lie algebra under the Lie bracket.
\section{Symmetries of system (2)}
In this section, we consider one parameter Lie group of infinitesimal transformations on ($x^{1}=x, x^{2}=t, u^{1}=a, u^{2}=b, u^{3}=c$, we use $x$ and $t$ instead of $x_{1}$ and $x_{2}$ respectively in order not to use index),
\begin{eqnarray}
\nonumber
\tilde{x}\ =\ x+\varepsilon\xi^1(x,t,a,b,c)+o(\varepsilon^2),\\\nonumber \tilde{t}\ =\ t+\varepsilon\xi^2(x,t,a,b,c)+o(\varepsilon^2),\hspace{1mm}\\\tilde{a}\ =\ a+\varepsilon\phi_1(x,t,a,b,c)+o(\varepsilon^2),\hspace{-.2mm}
\\\tilde{b}\ =\ b+\varepsilon\phi_2(x,t,a,b,c)+o(\varepsilon^2),\nonumber\\\tilde{c}\ =\ c+\varepsilon\phi_3(x,t,a,b,c)+o(\varepsilon^2).\nonumber\hspace{.7mm}
\end{eqnarray}
The associated symmetry generator is of the form:
\begin{eqnarray}
\nonumber
X=\xi^1(x,t,a,b,c)\partial_x+\xi^2(x,t,a,b,c)\partial_t+\hspace{2.5cm}\\
\phi_1(x,t,a,b,c)\partial_a+\phi_2(x,t,a,b,c)\partial_b+\phi_3(x,t,a,b,c)\partial_c.\hspace{-2.9mm}
\end{eqnarray}
and, it's second prolongation is the vector field:
\begin{eqnarray}
\nonumber
Pr^{(2)}X = X+\phi_1^x\partial_{a_x}+\phi_1^t\partial_{a_t}+\phi_2^x\partial_{b_x}+\phi_2^t\partial_{b_t}+\phi_3^x\partial_{c_x}+
\phi_3^t\partial_{c_t}+\\\phi_1^{xx}\partial_{a_{xx}}+\phi_2^{xx}\partial_{b_{xx}}+\phi_3^{xx}\partial_{c_{xx}}+\phi_1^{xt}\partial_{a_{xt}}+\phi_2^{xt}\partial_{b_{xt}}+\hspace{6mm}\\\phi_3^{xt}\partial_{c_{xt}}+\phi_1^{tt}\partial_{a_{tt}}+\phi_2^{tt}\partial_{b_{tt}}+\phi_3^{tt}\partial_{c_{tt}}.\hspace{32mm}\nonumber
\end{eqnarray}
Let (for convenience) $Q_1=\phi_1-\xi^1 a_x-\xi^2 a_t$, $Q_2=\phi_2-\xi^1b_x-\xi^2b_t$ and $Q_3=\phi_3-\xi^1c_x-\xi^2c_t$ then by using (7) we can compute the coefficients of (11):
\begin{eqnarray*}
\begin{array}{lclclclclcl}
\phi_1^x&=&D_xQ_1+\xi^1a_{xx}+\xi^2a_{xt},&&\phi_2^x&=&D_xQ_2+\xi^1b_{xx}+\xi^2b_{xt},\vspace*{1mm}\\
\phi_3^x&=&D_xQ_3+\xi^1c_{xx}+\xi^2c_{xt},&&\phi_1^t&=&D_tQ_1+\xi^1a_{xt}+\xi^2a_{tt},\vspace*{1mm}\\
\phi_2^t&=&D_tQ_2+\xi^1b_{xt}+\xi^2b_{tt},&&\phi_3^t&=&D_tQ_3+\xi^1c_{xt}+\xi^2c_{tt},\vspace*{1mm}\\
\phi_1^{xx}&=&D_x^2Q_1+\xi^1a_{xxx}+\xi^2a_{xxt},&&\phi_2^{xx}&=&D_x^2Q_2+\xi^1b_{xxx}+\xi^2b_{xxt},\vspace*{1mm}\\
\phi_3^{xx}&=&D_x^2Q_3+\xi^1c_{xxx}+\xi^2c_{xxt},&&\phi_1^{tt}&=&D_t^2Q_1+\xi^1a_{xtt}+\xi^2a_{ttt},\vspace*{1mm}\\
\phi_2^{tt}&=&D_t^2Q_2+\xi^1b_{xtt}+\xi^2b_{ttt},&&\phi_3^{tt}&=&D_t^2Q_3+\xi^1c_{xtt}+\xi^2c_{ttt},\vspace*{1mm}\\
\phi_1^{xt}&=&D_xD_tQ_1+\xi^1a_{xxt}+\xi^2a_{xtt},&&\phi_2^{xt}&=&D_xD_tQ_2+\xi^1b_{xxt}+\xi^2b_{xtt},\vspace*{1mm}\\
\phi_3^{xt}&=&D_xD_tQ_3+\xi^1c_{xxt}+\xi^2c_{xtt}.
\end{array}
\end{eqnarray*}
where $D_x$ and $D_t$ are the total derivatives with respect to $x$ and $t$ respectively. By (8) the invariance condition is equivalent with the following equations:
\begin{eqnarray*}
\begin{array}{lclcl}
\hspace*{-20mm}Pr^{(2)}X[a_{xx}-b_{tt}]=0,&&\ \ \ \ \ \ \ \ \ \ \ \ && a_{xx}-b_{tt}=0,\\
\hspace*{-20mm}Pr^{(2)}X[b_{xy}+c_{xx}]=0,&&\ \ \ \ \ \ \ \ \ \ \ \ \ &&b_{xy}+c_{xx}=0,\\
&&\hspace{10mm}\vdots\vspace{-7mm}
\end{array}
\end{eqnarray*}
\begin{eqnarray*}
\begin{array}{lclcl}
Pr^{(2)}X[a_{x}b_{x}-b_{x}c_{t}+b_{t}c_{x}-{c_{x}}^{2}+a_{}b_{xx}+2\,c_{}b_{xt}-b_{}c_{xt}]=0,\ \ \ \\\hspace*{25mm} a_{x}b_{x}-b_{x}c_{t}+b_{t}c_{x}-{c_{x}}^{2}+a_{}b_{xx}+2\,c_{}b_{xt}-b_{}c_{xt}=0.\
\end{array}
\end{eqnarray*}
After substituting $Pr^{(2)}X$, with coefficients, in the six equations above, we obtain six polynomial equations involving the various derivatives of $a$, $b$ and $c$, whose coefficients are certain derivatives of $\xi^1$, $\xi^2$, $\phi_1$, $\phi_2$ and $\phi_3$. Since, $\xi^1$, $\xi^2$, $\phi_1$, $\phi_2$ and $\phi_3$ depend only on $x$, $t$, $a$, $b$, $c$ we can equate the individual coefficients to zero, leading to the determining equations:
\begin{eqnarray*}
\begin{array}{lclcl}
a^3\xi^{1}_{a}=0, \ \ \ b^2\xi^{1}_{b}=0, \  \ \ a\xi^{2}_{b}=0, \ \ \
b^2\xi^{2}_{b}=0, \ \ \  ac\phi_{{2}_{a}}=0, \ \ \ cb\xi^{1}_{b}=0, \hdots\vspace*{3mm}\\
abc(-3\xi^{1}_t-\phi_{3_b})+2c^2a(-\xi^{2}_t+\xi^{1}_x)+ba^2(\phi_{2_b}-\phi_{3_c}-\xi^{1}_x+\xi^{2}_t)+\vspace*{1mm}\\
\hspace{4mm}ca^2(-2\xi^{2}_x+\phi_{{2}_c})+\phi_1(ab-2c^2)-a^2\phi_2+2ca\phi_3+4c^3\xi^{1}_t=0.
\end{array}
\end{eqnarray*}
We write some of these equation whereas the number of the whole is 410. By solving this system of PDEs, we can prove:
\begin{theorem}
The Lie group of point symmetries of system (2) has a Lie algebra generated by (10), whose coefficient functions are:
\begin{eqnarray}
\begin{array}{lclcl}
\phi_{1}=2{k_1}c+{k_7}\,a_{{}},&&\xi_{{1}} ={k_3}{x}+{k_1}{ t}+{k_2},\\

\phi_{{2}}=(2{k_4}-2
{k_3}+{k_7}) b_{{}}+2{k_6}\,c_{{}},&&\xi_{{2}}  ={k_6}{x}+{k_4}{t}+{k_5}.\\
\phi_{{3}} =( {k_7}+{k_4}-{k_3} ) c_{{}}+{k_6}a_{{}}+{k_1}b_{{}},
\end{array}
\end{eqnarray}
where $k_i$, $i=1,...,7$ are arbitrary constants.
\end{theorem}
\begin{cor}
Each one-parameter Lie group of point symmetries of (2) has following     infinitesimal generators:
 \begin{eqnarray}
\begin{array}{lclcl}
X_1=\partial_x,&&X_2=\partial_t,\ \ \ \ \ \ \ \ \ \ \ \ \ \ X_3=x\partial_x-2b\partial_b-c\partial_c,\\
 X_4=x\partial_t+2c\partial_b+a\partial_c,&&X_5=t\partial_x+2c\partial_a+b\partial_c\\X_6=t\partial_t+2b\partial_b+c\partial_c,&&X_7=a\partial_a+b\partial_b+c\partial_c.
\end{array}
\end{eqnarray}
The commutator table of Lie algebra for (2) is given below, where the entry in the $i^{th}$ row and $j^{th}$ column is $[X_{i},X_{j}]=X_{i}X_{j}-X_{j}X_{i}$, $i,j=1,...,7$.
\end{cor}
\begin{table}[h]
\begin{center}
\begin{tabular}{c|ccccccc}
\cline{1-8}
${[\,,\,]}$ & $X_1$  & $X_2$    & $X_3$    & $X_4$  & $X_5$ & $X_6$ & $X_7$  \\
\cline{1-8}
$X_1$ & $0$      & $0$    & $X_1$  & $X_2$    & $0$  & $0$ & $0$  \\
$X_2$ & $0$      & $0$    & $0$      & $0$       & $X_1$   & $X_2$ & $0$\\
$X_3$ & $-X_1$ & $0$     & $0$     & $X_4$    & $-X_5$ & $0$ & $0$\\
$X_4$ & $-X_2$ & $0$     & $-X_4$ & $0$    & $X_3-X_6+2X_7$ & $X_4$ & $0$    \\
$X_5$ & $0$      & $-X_1$& $X_5$  & $-X_3+X_6-2X_7$ & $0$&$-X_5$&$0$   \\
$X_6$ & $0$      & $-X_2$& $0$     & $-X_4$   & $X_5$ & $0$ & $0$ \\
$X_7$ & $0$      & $0$     & $0$     & $0$       & $0$  & $0$ & $0$\\
\cline{1-8}
\end{tabular} \\[5mm]
\end{center}
\end{table}
 The group transformation which is generated by $
X_i=\xi^1_i\partial_x+\xi^2_i\partial_t+
\phi_{1_i}\partial_a+\phi_{2_i}\partial_b+\phi_{3_i}\partial_c$ for $i=1,...,7$ is obtained by solving the seven systems of ODEs:
 \begin{eqnarray}
\begin{array}{lclcl}
 \frac{d\tilde{x}(s)}{ds}=\xi^1_i(\tilde{x}(s),\tilde{t}(s),\tilde{a}(s),\tilde{b}(s),\tilde{c}(s)),& \tilde{x}(0)=x,\vspace*{1mm}\\
 \frac{d\tilde{t}(s)}{ds}=\xi^2_i(\tilde{x}(s),\tilde{t}(s),\tilde{a}(s),\tilde{b}(s),\tilde{c}(s)),& \tilde{t}(0)=t,\vspace*{1mm}\\
 \frac{d\tilde{a}(s)}{ds}=\phi_{1_i}(\tilde{x}(s),\tilde{t}(s),\tilde{a}(s),\tilde{b}(s),\tilde{c}(s)),& \tilde{a}(0)=a,&&i=1,...,7\vspace*{1mm}\\
 \frac{d\tilde{b}(s)}{ds}=\phi_{2_i}(\tilde{x}(s),\tilde{t}(s),\tilde{a}(s),\tilde{b}(s),\tilde{c}(s)),& \tilde{b}(0)=b,\vspace*{1mm}\\
 \frac{d\tilde{c}(s)}{ds}=\phi_{3_i}(\tilde{x}(s),\tilde{t}(s),\tilde{a}(s),\tilde{b}(s),\tilde{c}(s)),& \tilde{c}(0)=c,
\end{array}
\end{eqnarray}
Exponentiating the infinitesimal symmetries (13), we obtain the one-parameter groups $g_k(s)$ generated by $X_k$, $k=1,...,7$:
 \begin{eqnarray}
\begin{array}{lclcl}
g_1(s)&:&(x,t,a,b,c)&\longmapsto&(x+s,t,a,b,c),\\ g_2(s)&:&(x,t,a,b,c)&\longmapsto&(x,t+s,a,b,c),\\ g_3(s)&:&(x,t,a,b,c)&\longmapsto&(xe^s,t,a,be^{-2s},ce^{-s}),\\ g_4(s)&:&(x,t,a,b,c)&\longmapsto&(x,sx+t,a,as^2+2cs+b,as+c),\\ g_5(s)&:&(x,t,a,b,c)&\longmapsto&(ts+x,t,bs^2+2cs+a,b,bs+c),\\ g_6(s)&:&(x,t,a,b,c)&\longmapsto&(x,te^s,a,be^{2s},ce^{s}),\\ g_7(s)&:&(x,t,a,b,c)&\longmapsto&(x,t,ae^s,be^s,ce^s),
\end{array}
\end{eqnarray}
Consequently, we can state the following theorem:
\begin{theorem}
If $a=f=f(x,t)$, $b=h=h(x,t)$ and $c=k=k(x,t)$ is a solution of (2), so are functions
 \begin{eqnarray*}
\begin{array}{lclcl}
g_1(s).f=f(x-s,t),&&g_1(s).h=h(x-s,t),&&\hspace*{-3mm}g_1(s).k=k(x-s,t),\\
g_2(s).f=f(x,t-s),&&g_2(s).h=h(x,t-s),&&\hspace*{-3mm}g_2(s).k=k(x,t-s),\\
g_3(s).f=f(xe^{-s},t),&&g_3(s).h=h(xe^{-s},t)e^{-2s},&&\hspace*{-3mm}g_3(s).k=k(xe^{-s},t)e^{-s},\\g_6(s).f=f(x,te^{-s}),&&g_6(s).h=h(x,te^{-s})e^{2s},&&\hspace*{-3mm}g_6(s).k=k(x,te^{-s})e^{s},\\g_7(s).f=f(x,t)e^{s},&&g_7(s).h=h(x,t)e^{s},&&\hspace*{-3mm}g_7(s).k=k(x,t)e^{s},\vspace{-3mm}
 \end{array}\end{eqnarray*}
 \begin{eqnarray*}
\begin{array}{lclcl}
\hspace{-27mm}g_4(s).f=f(x,t-sx),\\
\hspace{-27mm}g_4(s).h=f(x,t-sx)s^2+2k(x,t-sx)s+h(x,t-sx),\\
\hspace{-27mm}g_4(s).k=f(x,t-sx)s+k(x,t-sx),\vspace{-3mm}
 \end{array}\end{eqnarray*}
 \begin{eqnarray*}
\begin{array}{lclcl}
\hspace{-28.4mm}g_5(s).f=h(x-st,t)s^2+2k(x-st,t)s+f(x-st,t),\\
\hspace{-28.4mm}g_5(s).h=h(x-st,t),\\
\hspace{-28.4mm}g_5(s).k=h(x-st,t)s+k(x-st,t).
 \end{array}\end{eqnarray*}

\end{theorem}
This theorem is applied to obtain new solutions from known ones.
\begin{exam}
If we let
\begin{eqnarray*}
\begin{array}{lclcl}
a=f(x,t)=r_1x+r_2,&&c=k(x,t)=r_3x+r_4,\\
b=h(x,t)=\frac{r^2_3}{r_1}x-\frac{r_2r^2_3}{r^2_1}\ln(r_1x+r_2).
 \end{array}\end{eqnarray*}
be a simple solution of (2), where $r_1,r_2,r_3,r_4\in\mathbb{R}$ are arbitrary constants, we conclude that the functions $g_i(s).f(x,t), g_i(s).h(x,t)$ and $ g_i(s).k(x,t)$ are also solutions of (2) for $i=1,...,7$. for example
\begin{eqnarray*}
\begin{array}{lclcl}
g_5(s).f(x,t)=\vspace*{.5mm}\\
(\frac{r^2_3}{r_1}(x-ts)-\frac{r_2r^2_3}{r^2_1}\ln(r_1(x-ts)+r_2))s^2+2(r_3(x-ts)+r_4)s+r_1(x-ts)+r_2\vspace*{1mm}\\
g_5(s).h(x,t)=\frac{r^2_3}{r_1}(x-ts)-\frac{r_2r^2_3}{r^2_1}\ln(r_1(x-st)+r_2)\\
g_5(s).k(x,t)=(\frac{r^2_3}{r_1}(x-ts)-\frac{r_2r^2_3}{r^2_1}\ln(r_1(x-ts)+r_2))s+r_3(x-ts)+r_4
 \end{array}\end{eqnarray*}
is a set of new solutions  of (2) and we can obtain many other solutions by arbitrary combination of $g_i(s)$'s for $i=1,...,7$.
So we obtain infinite number of Einstein Walker manifolds just from this example.
\end{exam}
\section{One-dimensional optimal system of (2)}
In this section, we obtain the one-parameter optimal system of (2) by using symmetry group. Since every linear combination of infinitesimal symmetries is an infinitesimal symmetry, there is an infinite number of one-dimensional subalgebras for the differential equation. So it's important to determine which subgroups give different types of solutions. For this, we must find invariant solutions which are not related by transformation in the full symmetry group, this led to concept of an optimal system of subalgebra. For one-dimensional subalgebras, this classification problem is the same as the problem of classifying the orbits of the adjoint representation \cite{Olv1}. This problem is solved by the simple approach of taking a general element in the lie algebra and subjecting it to various adjoint transformation so as to simplify it as much as possible (\cite{Ovs} and \cite{Nad2}).
Optimal set of subalgebras is obtaining from taking only one representative from each class of equivalent subalgebras.
Adjoint representation of each $X_i$, $i=1,...,7$ is defined as follow:
\begin{equation}
\mathrm{Ad}(\exp(s.X_i).X_j) =
X_j-s.[X_i,X_j]+\frac{s^2}{2}.[X_i, [X_i,X_j]]-\cdots,
\end{equation}
where $s$ is a parameter and  $[X_i,X_j]$  is defined in corollary (3.2) for $i,j=1,\cdots,7$ (\cite{Olv1},page 199). If we denote by  $\mathfrak{g}$, the Lie algebra that  generated by (13), then we can write the adjoint action for $\mathfrak{g}$, and show that
\begin{theorem}
A one-dimensional optimal system of (2) is given by
\begin{eqnarray*}
\begin{array}{lclcl}
1)\ X_7,&&12)\ X_2+X_3+aX_5+bX_6+cX_7,\\
2)\ X_1+aX_7,&&13)\ -X_2+X_3+aX_5+bX_6+cX_7,\\
3)\ X_2+aX_7,&&14)\ X_3+X_4+aX_5+bX_6+cX_7,\\
4)\ X_6+aX_7\,&&15)\ X_3-X_4+aX_5+bX_6+cX_7,\\
5)\ X_1+X_6+aX_7,&&16)\ X_1+X_4+aX_5+bX_6+cX_7,\\
6)\ -X_1+X_6+aX_7,&&17)\ -X_1+X_4+aX_5+bX_6+cX_7,\\
7)\ X_5+aX_6+bX_7,&&18)\ X_2+X_3+X_4+aX_5+bX_6+cX_7,\\
8)\ X_2+X_5+aX_6+bX_7,&&19)\ -X_2+X_3+X_4+aX_5+bX_6+cX_7,\\
9)\ -X_2+X_5+aX_6+bX_7,&&20)\ X_2+X_3-X_4+aX_5+bX_6+cX_7,\\
\hspace*{-2mm}10)\ X_4+aX_5+bX_6+cX_7,&&21)\ -X_2+X_3-X_4+aX_5+bX_6+cX_7.\\
\hspace*{-2mm}11)\ X_3+aX_5+bX_6+cX_7,&&
\end{array}
\end{eqnarray*}
where $a,b,c\in{\Bbb R}$ are arbitrary numbers.
\end{theorem}
\begin{proof}
Attending to table of corollary (3.2), we understand that the center of $\mathfrak{g}$ is the subalgebra $\langle X_7\rangle$, so it is enough to specify the subalgebras of $\langle X_1,X_2,X_3,X_4,X_5,X_6\rangle$.
$F^s_i:\mathfrak{g}\to \mathfrak{g}$ defined by
$X\mapsto\mathrm{Ad}(\exp(sX_i).X)$ is a linear map, for
$i=1,\cdots,7$. The matrix $M^s_i$ of $F^s_i$, $i=1,\cdots,7$, with
respect to basis $\{X_1,\cdots,X_7\}$ is:
\begin{eqnarray*}
&\hspace{-6mm} M_1^s\!=\!\!\!\left[ \begin {array}{ccccccc}
1&0&0&0&0&0&0\\0&1&0&0&0&0&0\\-s&0&1&0&0&0&0\\0&-s&0&1&0&0&0\\0&0&0&0&1&0&0\\0&0&0&0&0&1&0\\0&0&0&0&0&0&1
\end{array} \right]\!\!,%
\;
M_4^s\!=\!\!\!\left[ \begin {array}{ccccccc}
1&s&0&0&0&0&0\\0&1&0
&0&0&0&0\\0&0&1&s&0&0&0\\0&0&0&1
&0&0&0\\0&0&-s&-{s}^{2}&1&s&-2s
\\0&0&0&-s&0&1&0\\0&0&0&0&0&0&1
\end {array} \right]\!\!,& \nonumber
\\
&\hspace{-5mm} M_2^s\!=\!\!\!\left[ \begin {array}{ccccccc}
1&0&0&0&0&0&0\\0&1&0
&0&0&0&0\\0&0&1&0&0&0&0\\0&0&0&1
&0&0&0\\-s&0&0&0&1&0&0\\0&-s&0&0
&0&1&0\\0&0&0&0&0&0&1
\end{array} \right]\!\!,%
\;
M_5^s\!=\!\!\!\left[ \begin {array}{ccccccc}
1&0&0&0&0&0&0\\s&1&0
&0&0&0&0\\0&0&1&0&-s&0&0\\0&0&s&
1&-{s}^{2}&-s&2\,s\\0&0&0&0&1&0&0
\\0&0&0&0&s&1&0\\0&0&0&0&0&0&1
\end{array} \right]\!\!,&
\\
&\hspace{-5mm} M_3^s\!=\!\!\!\left[ \begin {array}{ccccccc}
{{\rm e}^{s}}&0&0&0&0&0&0
\\0&1&0&0&0&0&0\\0&0&1&0&0&0&0
\\0&0&0&{{\rm e}^{-s}}&0&0&0\\0&0
&0&0&{{\rm e}^{s}}&0&0\\0&0&0&0&0&1&0
\\0&0&0&0&0&0&1
\end{array} \right]\!\!,%
\;
M_6^s\!=\!\!\!\left[ \begin {array}{ccccccc}
1&0&0&0&0&0&0\\0&{
{\rm e}^{s}}&0&0&0&0&0\\0&0&1&0&0&0&0
\\0&0&0&{{\rm e}^{s}}&0&0&0\\0&0
&0&0&{{\rm e}^{-s}}&0&0\\0&0&0&0&0&1&0
\\ 0&0&0&0&0&0&1
\end{array} \right]\!\!.&\nonumber
\end{eqnarray*}
respectively and $M^s_7$ is the $7\times 7$ identity matrix. Let
$X=\sum_{i=1}^7a_iX_i$, then
\begin{eqnarray*}
 \hspace{-20mm}F^{s_7}_7\circ \cdots \circ F^{s_1}_1 \ : \ X \mapsto \ ((1+s_5s_4)e^{s_3}a_1+e^{s_6+s_3}s_4a_2).X_1+\cdots \hspace{20mm}\\
 \hspace{20mm}+(-s_5s_2a_1-e^{s_6}s_2a_2+\cdots+(1+s_5s_4)a_6-2s_5s_4a_7)X_6+a_7X_7.\vspace*{2mm}
\end{eqnarray*}

Now, we can simplify $X$ as follows: \vspace*{1mm}

If $a_1=a_2=\cdots=a_6=0$, then $X$ is reduced to the Case of (1), that is center of $\mathfrak{g}$.\vspace*{1mm}

If $a_3=\cdots=a_6=0$ and $a_1\neq0$, then we can make the coefficient of
$X_2$ vanish by $F^{s_4}_4 $; By setting $s_4=a_2/a_1$. Scaling $X$ if necessary, we can assume that $a_1=1$. So, $X$ is reduced to the Case of (2).\vspace*{1mm}

If $a_1=a_3=\cdots=a_6=0$, then we can make $a_2=\pm 1$ by $F^{s_6}_6 $; By setting $s_6=\ln{\left| {a_2} \right|}$. So, $X$ is reduced to  Case (3).
\vspace*{1mm}

If $a_1=a_3=\cdots=a_5=0$ and $a_6\neq0$, then we can make the coefficient of
$X_2$ vanish by $F^{s_2}_2 $; By setting $s_2=-a_2/a_6$. Scaling $X$ if necessary, we can assume that $a_6=1$. So, $X$ is reduced to the Case of (4).\vspace*{1mm}

If $a_3=\cdots=a_5=0$ and $a_6\neq0$, then we can make the coefficient of
$X_2$ vanish by $F^{s_2}_2 $; By setting $s_2=-a_2/a_6$. Scaling $X$ if necessary, we can assume that $a_6=1$. then, we can make $a_1=\pm 1$, by $F^{s_3}_3$; By setting  $s_3=\ln{\left| {a_1} \right|}$. So, $X$ is reduced to the Cases of (5) and (6).\vspace*{1mm}

If $a_2=\cdots=a_4=0$ and $a_5\neq0$, then we can make the coefficient of
$X_1$ vanish by $F^{s_2}_2 $; By setting $s_2=-a_1$. Scaling $X$ if necessary, we can assume that $a_5=1$. So, $X$ is reduced to the Case of (7).\vspace*{1mm}

If $a_3=a_4=0$ and $a_5\neq0$, then we can make the coefficient of
$X_2$ vanish by $F^{s_2}_2 $; By setting $s_2=-a_1$. Scaling $X$ if necessary, we can assume that $a_5=1$. then we can make $a_2=\pm 1$, by $F^{s_6}_6$; By setting  $s_6=\ln{\left| {a_2} \right|}$. So, $X$ is reduced to the Cases of (8) and (9).\vspace*{1mm}

If $a_1=a_3=0$ and $a_4\neq0$, then we can make the coefficient of
$X_2$ vanish by $F^{s_1}_1 $; By setting $s_1=-a_2/a_4$. Scaling $X$ if necessary, we can assume that $a_4=1$. So, $X$ is reduced to the Case of (10).\vspace*{1mm}

If $a_2=a_4=0$ and $a_3\neq0$, then we can make the coefficient of
$X_1$ vanish by $F^{s_1}_1$; By setting $s_1=-a_1$. Scaling $X$ if necessary, we can assume that $a_3=1$. So, $X$ is reduced to the Case of (11).\vspace*{1mm}

If $a_4=0$ and $a_3\neq0$, then we can make the coefficient of
$X_1$ vanish by $F^{s_1}_1$; By setting $s_1=-a_1$. Scaling $X$ if necessary, we can assume that $a_3=1$. then we can make $a_2=\pm 1$ by $F^{s_3}_3$; By setting $s_3=-\ln{\left| {a_4} \right|}$ .So, $X$ is reduced to the Cases of (12) and (13).\vspace*{1mm}

If $a_2=0$ and $a_3\neq0$, then we can make the coefficient of
$X_1$ vanish by $F^{s_1}_1$; By setting $s_1=-a_1$. Scaling $X$ if necessary, we can assume that $a_3=1$. then we can make $a_4=\pm 1$ by $F^{s_6}_6$; By setting $s_6=-\ln{\left| {a_2} \right|}$ .So, $X$ is reduced to the Cases of (14) and (15).\vspace*{1mm}

If $a_3=0$ and $a_4\neq0$, then we can make the coefficient of
$X_2$ vanish by $F^{s_1}_1$; By setting $s_1=-a_2/a_4$. Scaling $X$ if necessary, we can assume that $a_4=1$. then we can make $a_1=\pm 1$ by $F^{s_3}_3$; By setting $s_3=\ln{\sqrt{a_1}}$ .So, $X$ is reduced to the Cases of (16) and (17).\vspace*{1mm}

If $a_3\neq0$, then we can make the coefficient of
$X_1$ vanish by $F^{s_1}_1$; By setting $s_1=-a_1$. Scaling $X$ if necessary, we can assume that $a_3=1$. then we can make $a_2=\pm 1$ and $a_4=\pm 1$ by $F^{s_3}_3$ and $F^{s_6}_6$; By setting $s_3=-\ln{\left| {a_4} \right|}$  and $s_6=\ln{\left| {a_2} \right|}$ respectively. So, $X$ is reduced to the Cases of (18), (19), (20) and (21).
\end{proof}
\section{Similarity reduction of system (2)}
The system (2) is expressed in the coordinates $(x,t,a,b,c)$, so we must search for this system's form in specific coordinates in order to reduce it. Those coordinates will be constructed  by searching for independent  invariants $w,f,h,k$ corresponding to the infinitesimal generator. Hence by using the chain rule, the expression of the system in the new coordinate leads to the reduced system.
Now, we find some of nontrivial solution of system (2).

First, consider $X_3=x\partial_x-2b\partial_b-c\partial_c$. For determining independent invariants $I$, we ought to solve the first PDEs $X_i(I)=0$, that is
\begin{equation*}
(x\partial_x-2b\partial_b-c\partial_c)I=x\frac{\partial I}{\partial x}+0\frac{\partial I}{\partial t}+0\frac{\partial I}{\partial a}-2b\frac{\partial I}{\partial b}-c\frac{\partial I}{\partial c},
\end{equation*}
which is a homogeneous first order PDE. So we must solve the associated characteristic ODE
\begin{equation*}
\frac{dx}{x}=\frac{dt}{0}=\frac{da}{0}=\frac{db}{-2b}=\frac{dc}{-c}.
\end{equation*}
Hence, we obtain four functionally independent invariants $w=t$, $f=a$, $h=bx^{2}$ and $k=cx$.
If we treat $f$, $h$ and $k$ as functions of $w$, we can compute formulae for the derivatives of $a$, $b$ and $c$ with respect to $x$ and $t$ in terms of $w$, $f$,  $h$, $k$ and the derivatives of $f$,  $h$ and $k$ with respect to $w$. We have $a=f(w)=f(t)$, $b=x^{-2}h(w)=x^{-2}h(t)$ and $c=x^{-1}k(w)=x^{-1}k(t)$. So by using the chain rule we have,
 \begin{eqnarray*}
\begin{array}{lclcl}
a_1=a_x=\frac{\partial a}{\partial x}=\frac{\partial f}{\partial x}=\frac{\partial f}{\partial w}\frac{\partial w}{\partial x}=\frac{\partial f}{\partial t}\frac{\partial t}{\partial x}=0,&&a_{11}=a_{xx}=0,\\
a_2=a_t=\cdots =f_t,&&a_{22}=f_{tt},\\
b_1=-2x^{-3}h,&&b_{11}=6x^{-4}h,\\
b_2=x^{-2}h_t,&&b_{22}=x^{-2}h_{tt},\\
c_1=-x^{-2}k,&&c_{11}=2x^{-3}k,\\
c_2=x^{-1}k_t,&&c_{22}=x^{-1}k_{tt},\\
a_{12}=0,&&b_{12}=-2x^{-3}h_t,\\
c_{12}=-x^{-2}k_t.
\end{array}\end{eqnarray*}

Substituting these in the system (2) we obtain the reduced system,
 \begin{eqnarray*}
\begin{array}{lclcl}
h_{tt}=0,&&&&f_th_t+f_tk-k^2_t+hf_{tt}+fk_t=0,\\
h_t-k=0,&&&&2f_th-2kk_t+2fk+hk_{tt}=0,\\
k_{tt}=0,&&&&3hk_t-5h_tk-k^2+6fh=0.
\end{array}\end{eqnarray*}

that is a system of ODE. By solving this system we obtain two types of solutions

\begin{equation*}
1)\left \{ \begin{array}{lcr} f = f(t)\\ h = 0  \\ k = 0 \\ \end{array} \right.
\hspace*{20mm}  2)\left \{ \begin{array}{lcr} f = \frac{r^2_2}{r_2t+r_1}\\ h =r_2t+r_1  \\ k =r_2 \\  \end{array} \right.
\end{equation*}

where $r_1$ and $r_2$ are arbitrary constants in type (2) and $f(t)$ is an arbitrary function in type (1). We can compute all of invariant solutions for other symmetry generators in a similar way. Some of  infinitesimal symmetries and  their Lie invariants are listed in the following Table.
\begin{table}[h]
\begin{center}
Table I
\begin{tabular}{c|llllllll}
\cline{1-9}
$i$ & $V_i$  & $w_i$    & $f_i(w)$    & $h_i(w)$  & $k_i(w)$ & $a_i$ & $b_i$ & $c_i$   \\
\cline{1-9}
$1$ & $X_1$ & $t$    & $a$  & $b$    & $c$  & $f(w)$ & $h(w)$  & $k(w)$  \\
$2$ & $X_2$ & $x$   & $a$  & $b$    & $c$  & $f(w)$ & $h(w)$  & $k(w)$\\
$3$ & $X_6$ & $x$   & $a$  & $bt^{-2}$    & $ct^{-1}$  & $f(w)$ & $t^2h(w)$  & $tk(w)$\\
$4$ & $X_1+X_7$ & $t$ & $ae^{-x}$  & $be^{-x}$    & $ce^{-x}$  & $e^{x}f(w)$ & $e^{x}h(w)$  & $e^{x}k(w)$    \\
$5$ & $X_2+X_7$ & $x$ & $ae^{-t}$  & $be^{-t}$    & $ce^{-t}$  & $e^{t}f(w)$ & $e^{t}h(w)$  & $e^{t}k(w)$\\
$6$ & $X_6+X_7$ & $x$ & $at^{-1}$  & $bt^{-3}$    & $ct^{-2}$  & $tf(w)$ & $t^3h(w)$  & $t^2k(w)$ \\
$7$ & $X_1+X_6+X_7$ & $te^{-x}$ & $ae^{-x}$  & $be^{-3x}$    & $ce^{-2x}$  & $e^{x}f(w)$ & $e^{3x}h(w)$  & $e^{2x}k(w)$\\
\cline{1-9}
\end{tabular} \\[5mm]
\end{center}
\end{table}

Now we list the reduced form of system (2) corresponding to infinitesimal generators and obtain some of its invariant solutions.
Reduced system for case1)
 \begin{eqnarray*}
\begin{array}{lclclclclclcl}
h_{tt}=0,&&& k_{tt}=0,&&& f_th_t-k^2_t+hf_{tt}=0,&&&hk_{tt}=0.\\
\end{array}\end{eqnarray*}

Invariant solutions for case1)
\begin{eqnarray*}
\begin{array}{lclcl}
1)\ \left\{ \begin{gathered}
  f(t)= f(t) \hfill \\
  h(t)= 0 \hfill \\
  k(t)=r_1 \hfill \\
\end{gathered}  \right. , && 2)\ \left\{\begin{gathered}f(t)=\frac{r^2_5t}{r_3}+(\frac{r_1r_3-r_4r^2_5}{r^2_3})\ln(r_3t+r_4)+r_2 \hfill \\
  h(t)=r_3t+r_4 \hfill \\
  k(t)=r_5t+r_6 \hfill \\
\end{gathered} \right.
 \end{array}\end{eqnarray*}

Reduced system for case2)
 \begin{eqnarray*}
\begin{array}{lclclclclclc}
f_{xx}=0,&&k_{xx}=0,&& f_xh_x-k^2_x+fh_{xx}=0,&& kf_{xx}-fk_{xx}=0.
\end{array}\end{eqnarray*}

Invariant solutions for case2)
\begin{eqnarray*}
\begin{array}{lclcl}
1)\ \left\{ \begin{gathered}
  f(x)=0 \hfill \\
  h(x)=h(x) \hfill \\
 k(x)=r_1 \hfill \\
\end{gathered}  \right. , && 2)\ \left\{\begin{gathered}f(x)=r_3x+r_4 \hfill \\
  h(x)=\frac{r^2_5x}{r_3}+(\frac{r_1r_3-r_4r^2_5}{r^2_3})\ln(r_3x+r_4)+r_2 \hfill \\
  k(x)=r_5x+r_6 \hfill \\
\end{gathered} \right.
 \end{array}\end{eqnarray*}

Reduced system for case3)
\begin{eqnarray*}
\begin{array}{lclcl}
f_{xx}-2h=0,&& 2k_xk-kf_{xx}+fk_{xx}=0\\
2h_x+k_{xx}=0,&& f_xk-k^2-fk_x=0\\
-f_xh_x-3h_xk-hk_x+k^2_x-fh_{xx}=0,\\
\end{array}\end{eqnarray*}

Invariant solutions for case3)
\begin{eqnarray*}
\begin{array}{lclcl}
1)\left\{ \begin{gathered}
k=0  \hfill \\
f=\frac{r_1x^2}{2}+r_2x+r_3 \hfill \\
h=\frac{r_1}{2}\hfill \\
\end{gathered}  \right. , 2)\left\{\begin{gathered}                                   k=r_2x+r_3 \hfill \\
f=k(x+r_1)\hfill \\
h=r_2\hfill \\
\end{gathered} \right. ,3)\left\{ \begin{gathered}
k=\frac{128-r_1}{3(x+r_2)^2}+r_3x+r_4\hfill \\
f=-\frac{4kk_{xx}}{k_{xxx}} \hfill \\                          h=\frac{k_xk_{xxx}-2k^2_{xx}}{k_{xxx}}\hfill \\
\end{gathered}  \right .
 \end{array}\end{eqnarray*}

Reduced system for case4)
 \begin{eqnarray*}
\begin{array}{lclcl}
h_{tt}-f=0,&& f_th_t+f_tk-k^2_t+hf_{tt}=0,&& h_t+k=0,\\
f_t+k_{tt}=0,&&hk_{tt}-f_th+2kk_t=0,&&2fh-2hk_t+3h_tk-k^2=0.\\
\end{array}\end{eqnarray*}

Invariant solutions for case4)
\begin{eqnarray*}
\begin{array}{lclcl}
1)\left\{ \begin{gathered}
k(t)=0  \hfill \\
h(t)=0 \hfill \\
f(t)=0\hfill \\
\end{gathered}  \right. ,&& 2)\left\{\begin{gathered}                                            k(t)=0 \hfill \\
h(t)=r_1\hfill \\
f(t)=0\hfill \\
\end{gathered} \right. ,&&3)\left\{ \begin{gathered}
k(t)=r_2e^{r_1t}\hfill \\
h(t)=-\frac{r_2}{r_1}e^{r_1t}\hfill \\                                                                     f(t)=-r_1r_2e^{r_1t}\hfill \\
\end{gathered}  \right .
 \end{array}\end{eqnarray*}

Reduced system for case5)
 \begin{eqnarray*}
\begin{array}{lclcl}
f_{xx}-h=0,&&2hf-2fk_x+3f_xk-k^2 =0\\
h_x+k_{xx}=0,&&fh_x-2k_xk+kf_{xx}-fk_{xx}-hk=0\\
f_x+k=0,&&h_xf_x+h_xk-k^2_x+fh_{xx}=0
\end{array}\end{eqnarray*}

Invariant solutions for case5)
 \begin{eqnarray*}
\begin{array}{lclcl}
1)\left\{ \begin{gathered}
k(x)=0  \hfill \\
f(x)=0 \hfill \\
h(x)=0\hfill \\
\end{gathered}  \right. ,&& 2)\left\{\begin{gathered}                                            k(x)=0 \hfill \\
f(x)=r_1\hfill \\
h(x)=0\hfill \\
\end{gathered} \right. ,&&3)\left\{ \begin{gathered}
k(x)=r_2e^{r_1x}\hfill \\
f(x)=-\frac{r_2}{r_1}e^{r_1x}\hfill \\                                                                     h(x)=-r_1r_2e^{r_1x}\hfill \\
\end{gathered}  \right .
 \end{array}\end{eqnarray*}

 Reduced system for case6)
 \begin{eqnarray*}
\begin{array}{lclcl}
f_{xx}-6h=0,&&-4f_xk-3fh+3fk_x+4k^2 =0\\
3h_x+k_{xx}=0,&&-fh_x+4k_xk-kf_{xx}+fk_{xx}+2hk=0\\
f_x+2k=0,&&f_xh_x+4h_xk+hk_x-k^2_x+fh_{xx}=0
\end{array}\end{eqnarray*}

Invariant solutions for case6)
\begin{eqnarray*}
\begin{array}{lclcl}
1)\ \left\{ \begin{gathered}
k(x)=0 \hfill \\
h(x)=0 \hfill \\
f(x)=r_1 \hfill \\
\end{gathered}  \right. , && 2)\ \left\{\begin{gathered}                            k(x)=-\frac{27}{(r_1x+r_2)^3} \hfill \\
h(x)=-\frac{1}{3}k_x \hfill \\
f(x)=-\frac{3k^2}{k_x} \hfill \\
\end{gathered} \right.
 \end{array}\end{eqnarray*}

Reduced system for case7)
 \begin{eqnarray*}
\begin{array}{lclcl}
{h''}w-2{h'}-{k''}w^2+3{k'}w-4k=0,&&\hspace{-3mm}{f''}w^2-{f'}w+f-{h''}=0 \\
 {f'}{h'}-2{f'}k-{k'}^2-2k{f''}w+h{f''}+f{k''}w=0,&&\hspace{-3mm} {k''}-{f''}w=0\\
{f'}(3h-{h'}w)+k(-3{k'}+{f''}w^2-{f'}w-3f)-\\
(fw^2-kw+h){k''}+{k'}^2w+3f{k'}w=0\\
{f'}({h'}w^2-3hw)-6f({h'}w-2h)+6{h'}k-{k'}^2w^2+\\
4{k'}(kw-h)-4k^2+{h''}(fw^2-2kw)+h{k''}w=0
\end{array}\end{eqnarray*}

where $f',\cdots$ are derivatives with respect to $w$ and $r_i$'s are arbitrary constants.

We can construct reduced systems for all generators in Theorem 4.1 and obtain their invariant solutions. Then by substituting invariant solutions in Table I, we obtain corresponding solutions for system (2). For example consider the invariant solution number (3) for case 5. by using Table I, we obtain solutions
  \begin{equation*}
\begin{array}{lclcl}
a=-\frac{r_2}{r_1}e^{r_1x}e^t,&&b=-r_1r_2e^{r_1x}e^t,&&c=r_2e^{r_1x}e^t.
\end{array}\end{equation*}

for system (2). Consequently we can determine the general form of Einstein-Walker manifolds for case 5, by substituting $a$, $b$ and $c$ in (1). In addition we can construct many other solutions from this own by using Theorem 3.3, as we did in example 3.4.

\section*{Conclusion}
In this paper, by applying the method of Lie symmetries, we find the Lie point symmetries group of system(2). Also, we have obtained the one-parameter optimal system of subalgebras for (2).Then the Lie invariants and similarity reduced systems and some of the invariant solutions are obtained. In conclusion, we specified a large class of Einstein-Walker manifolds.


%
\end{document}